\documentclass{amsart}

\usepackage{amssymb,latexsym,amsmath}

\newtheorem{theorem}{Theorem}[section]
\newtheorem{lemma}[theorem]{Lemma}

\theoremstyle{definition}

\theoremstyle{remark}
\newtheorem{remark}[theorem]{Remark}

\numberwithin{equation}{section}

\newcommand{\abs}[1]{\lvert#1\rvert}

\begin{document}

\title{The $L^p$ boundedness of the Bergman projection for a class of bounded Hartogs domains}

\author{Liwei Chen}
\address{Department of Mathematics, Washington University in St. Louis, St. Louis, Missouri 63130}
\email{chenlw@math.wustl.edu}
\thanks{Advised by Steven Krantz}


\date{Jan 09, 2014}


\begin{abstract}
We generalize the Hartogs triangle to a class of bounded Hartogs domains, and we prove that the corresponding Bergman projection is bounded on $L^p$ if and only if $p$ is in the range $(\frac{2n}{n+1},\frac{2n}{n-1})$.
\end{abstract}

\maketitle

\section{Introduction}
Regularity properties of the Bergman projection have been studied for many years, in particular, its $L^p$ boundedness is of considerable interest. When the underlying domain has sufficiently smooth boundary, decisive results were obtained by several people working on it (e.g. \cite{F}, \cite{PS}, \cite{M1}, \cite{NRSW}, \cite{M2}, \cite{MS}, \cite{CD}, \cite{BS}, etc.). Most recently, in \cite{LS2}, Lanzani and Stein show that the Bergman projection is bounded on $L^p$ for $1<p<\infty$, when the underlying domain is strongly pseudoconvex with only $C^2$ boundary. However, when we look at some domains with serious singularities at their boundaries, the $L^p$ boundedness of the Bergman projection will no longer hold for all $p\in(1,\infty)$.\\

There is also a sequence of papers concerning domains with non-smooth boundary. In \cite{LS1}, Lanzani and Stein consider non-smooth planar domains, and the ranges for $p$ depend on different types of boundaries. In a series of papers, \cite{KP1} and \cite{KP2}, Krantz and Peloso show that the Bergman projection for the two dimensional non-smooth worm domain is bounded only when $p$ is in a range depending on the winding of the domain. In \cite{CS1} and \cite{CS2}, Chakrabarti and Shaw focus on the $\overline{\partial}$-equation and the corresponding Sobolev regularities, over the product domains and the Hartogs triangle. Recently, in \cite{Z}, Zeytuncu applies Forelli-Rudin's inflation idea to consider a class of domains of the form $\{(z,w)\in\mathbb{C}^2|\ z\in\mathbb{D},\ \abs{w}^2<\mu(z)\}$, where $\mu$ is a weight on $\mathbb{D}$. We should point out that this class of domains is slightly different from what we focus here. \\

As a well known fact, the Hartogs triangle does not possess a Stein neighborhood bases, and the Sobolev regularity of its Bergman projection does not behave well near the nontrivial singularity on the boundary. Hence, we may expect the $L^p$ regularity of the Bergman projection of the Hartogs triangle and its generalization will not behave well, too. In particular, here we figure out the exact range of $p$ for the $L^p$ boundedness of the corresponding Bergman projection. \\
  
In this paper, we study a class of bounded Hartogs domains which generalize the two dimensional Hartogs triangle $\mathbb{H}=\{(z_1,z_2)\in\mathbb{C}^2|\ \abs{z_1}<\abs{z_2}<1\}$. To be precise, for $j=1,\dots,l$, let $\Omega_j$ be a bounded smooth domain in $\mathbb{C}^{k_j}$, let $\phi_j:\Omega_j\to\mathbb{B}^{k_j}$ be a biholomorphic mapping between $\Omega_j$ and the unit ball $\mathbb{B}^{k_j}$ in $\mathbb{C}^{k_j}$ with inverse $\phi^{-1}_j$, let $m_j=\sum_{s=1}^j k_s$ with $m_0=0$ and $m_l=k<n$, and we use the notation that $\tilde{z_j}=(z_{m_{j-1}+1},\dots,z_{m_j})$ when $z\in\mathbb{C}^n$. For $1\le k<n$, we look at the bounded Hartogs domain given by $$\mathbb{H}^n_{\{k_j,\phi_j\}}=\{z\in\mathbb{C}^n|\max_{1\le j \le l}\abs{\phi_j(\tilde{z_j})}<\abs{z_{k+1}}<\cdots<\abs{z_{n}}<1\}$$ in $\mathbb{C}^n$. When $l=1$, $k=1$, $n=2$, and $\phi_1$ is the identity map, we obtain the classical Hartogs triangle.

For nontrivial examples, we can take $\phi_j$ to be nonsingular linear mappings. When $n=4$, $l=2$, $k_1=1$, $k_2=2$, $\phi_1(z_1)=2z_1-1$, $\phi_2(z_2,z_3)=(z_2+\frac{1}{2}z_3,z_3)$, we obtain a bounded domain which is the intersection of two unbounded domains,
$$\{z\in\mathbb{C}^4|\ \abs{2z_1-1}^2<\abs{z_4}^2<1,\ \abs{z_2+\frac{1}{2}z_3}^2+\abs{z_3}^2<\abs{z_4}^2<1\}.$$
We can also take $\phi_j$ to be nonlinear, then we may obtain other types of domains. When $n=3$, $l=1$, $k=k_1=2$, $\phi_1(z_1,z_2)=(\frac{z_1}{z_2-10},3z_2+1)$, the domain becomes
$$\{z\in\mathbb{C}^3|\ \abs{\frac{z_1}{z_2-10}}^2+\abs{3z_2+1}^2<\abs{z_3}^2<1\}.$$
By this consideration, the domains $\mathbb{H}^n_{\{k_j,\phi_j\}}$ can be a large class of doamins to some extent.\\

For an arbitrary domain $D$, the corresponding Bergman projection $P_D$ is originally defined on $L^2(D)$, mapping onto the Bergman space $A^2(D)=\mathcal{O}(D)\cap L^2(D)$. If $D$ is bounded, then $L^{p'}(D)\subset L^p(D)$ when $p'>p$. So the Bergman projection $P_D$ can be defined on $L^p(D)\cap L^2(D)=L^p(D)$ for all $p\ge 2$. When $1\le p<2$, the Bergman projection $P_D$ on $L^p(D)$ will mean $P_D$ on the subspace $L^p(D)\cap L^2(D)$ of $L^p(D)$. So for any $p\in [1,\infty)$, when we say the Bergman projection $P_D$ is bounded on $L^p(D)$, we mean the Bergman projection $P_D$ mapping $L^p(D)\cap L^2(D)$ onto $A^p(D)\cap L^2(D)$ is bounded, where $A^p(D)=\mathcal{O}(D)\cap L^p(D)$. With all these in mind, we can state our main result.

\begin{theorem}
For $1\le p<\infty$ and $1\le k<n$, the Bergman projection $P_{\mathbb{H}^n_{\{k_j,\phi_j\}}}$ for $\mathbb{H}^n_{\{k_j,\phi_j\}}$ is bounded on $L^p(\mathbb{H}^n_{\{k_j,\phi_j\}})$ if and only if $p$ is in the range $(\frac{2n}{n+1},\frac{2n}{n-1})$.
\end{theorem}

It is quite interesting that the boundedness range for $p$ does not depend on $\{k_j,\phi_j\}$, but only on the dimension $n$. If we take $\phi_j=id_j$, the identity map of $\mathbb{B}^{k_j}$, then we will obtain our standard model 
$$\mathbb{H}^n_{\{k_j\}}=\{z\in\mathbb{C}^n|\max_{1\le j \le l}\abs{\tilde{z_j}}<\abs{z_{k+1}}<\cdots<\abs{z_n}<1\},$$
which plays an important role in this article. For simplicity, we will mainly deal with the case $l=1$, since the general case requires no new work but clumsy notations. So we come to the $n$-dimensional generalized Hartogs triangle $$\mathbb{H}^n_{k}=\{z\in\mathbb{C}^n|\ \abs{(z_1,\dots,z_k)}<\abs{z_{k+1}}<\cdots<\abs{z_{n}}<1\}.$$ In order to emphasize our standard model, we state the special case separately as follows.

\begin{theorem}
For $1\le k<n$, the Bergman projection $P_{\mathbb{H}^n_{k}}$ for $\mathbb{H}^n_{k}$ is bounded on $L^p(\mathbb{H}^n_{k})$ when $\frac{2n}{n+1}<p<\frac{2n}{n-1}$, and is unbounded when $1\le p\le \frac{2n}{n+1}$. The same conclusion is true for our standard model $\mathbb{H}^n_{\{k_j\}}$.
\end{theorem}

\begin{remark}
Indeed, for $p\in (\frac{2n}{n+1},\frac{2n}{n-1})$, by a simple argument, we can extend the Bergman projection $P_{\mathbb{H}^n_{k}}$ to all the space  $L^p(\mathbb{H}^n_{k})$, not just defined on the subspace $L^p(\mathbb{H}^n_{k})\cap L^2(\mathbb{H}^n_{k})$. The same is true for $\mathbb{H}^n_{\{k_j\}}$ and $\mathbb{H}^n_{\{k_j,\phi_j\}}$.
\end{remark}

\begin{remark}
For $l=1$, when $n=k$, the generalized Hartogs triangle will degenerate to the unit ball $\mathbb{B}^k$ in $\mathbb{C}^k$. A classical result states that the Bergman projection is bounded on $L^p$ for $p\in (1,\infty)$. Moreover, the domain $\mathbb{H}^k_{k,\phi_1}$ will degenerate to the smoothly bounded domain $\Omega$, which indeed is strongly pseudoconvex. Therefore the $L^p$ boundedness of the Bergman projection is true for $p\in (1,\infty)$ by the result in \cite{LS2}. In fact, if we follow the method in the proof of theorem 1.1, for a general $l$, when $n=k$, we can see that the same result holds for the degenerate domain.
\end{remark}

Our starting point is to study the Bergman projection for the classical Hartogs triangle by following the method used in \cite{HKZ}, where they only deal with the unit disk in $\mathbb{C}$. After applying a biholomorphism to transfer the Hartogs triangle to a product domain, we find that we can use a similar estimate (compared to the unit disk) to control the absolute value of the determinant of the Jaccobian from the transformation, but we have to make a restriction on $p$, namely $p\in(\frac{4}{3},4)$. It is a little bit tricky to find a bounded sequence $\{f_m\}$ in $L^p\cap L^2$, such that the sequence $\{P_{\mathbb{H}}(f_m)\}$ blows up in $L^p$, when $p=\frac{4}{3}$. Further, we find that this idea can be generalized to higher dimensions, so we set up the $n$-dimensional generalized Hartogs triangle $\mathbb{H}^n_{k}$ and the standard model $\mathbb{H}^n_{\{k_j\}}$. Finally, by Bell's extension theorem (see \cite{B}), we find that the $L^p$ boundedness of the Bergman projection will be preserved for a more general class of bounded Hartogs domains $\mathbb{H}^n_{\{k_j,\phi_j\}}$.

\subsection*{Acknowledgements}
The author thanks his thesis advisor Prof. S. Krantz for very helpful comments and suggestions on his research, and the support of typing this paper.\\

\section{transfer to a product domain}
We start the proof by transferring the $n$-dimensional generalized Hartogs triangle to a product domain.
\subsection{The biholomorphism}
Let $\mathbb{D}$ be the unit disk in $\mathbb{C}$, and let $\mathbb{D}^*=\mathbb{D}\backslash\{0\}$ be the punctured disk. Let us look at the $n$-dimensional generalized Hartogs triangle 
$$\mathbb{H}^n_{k}=\{z\in\mathbb{C}^n|\ \abs{(z_1,\dots,z_k)}<\abs{z_{k+1}}<\cdots<\abs{z_{n}}<1\}.$$
Define a map $F: \mathbb{H}^n_{k}\to \mathbb{B}^k\times\mathbb{D}^*\times\cdots\times\mathbb{D}^*$ ($n-k$ copies of $\mathbb{D}^*$) given by 
$$F(z_1,\dots,z_n)=\Big(\frac{z_1}{z_{k+1}},\dots,\frac{z_k}{z_{k+1}},\frac{z_{k+1}}{z_{k+2}},\dots,\frac{z_{n-1}}{z_n},z_n\Big),$$
which can be easily seen to be a biholomorphism with inverse
$$G(w_1,\dots,w_n)=\big(w_1(w_{k+1}\cdots w_n),\dots,w_k(w_{k+1}\cdots w_n),(w_{k+1}\cdots w_n),\dots,(w_{n-1}w_n),w_n\big).$$
By a direct computation, we see that the determinant of the Jaccobian of $G$ is given by $\det\mathcal{J}_{G}^{\mathbb{C}}(w)=w_{k+1}^k\cdots w_n^{n-1}$.

\subsection{The Bergman kernels}
Consider the punctured disk $\mathbb{D}^{*}$, any function $f\in\mathcal{O}(\mathbb{D}^*)$ has a Laurent expansion $f(w)=\sum_{j=-\infty}^{\infty}a_jw^j$ converging uniformly on compact subsets. A direct computation shows that $w^j\in L^2(\mathbb{D}^*)$ when $j\ge 0$, and they are orthogonal. Hence, it follows easily that the set $\{w^j\}_{j\ge0}$ is a complete basis in the Bergman space $A^2(\mathbb{D}^*)$. If we normalize the area measure, namely $dV(w)=\frac{1}{\pi}dudv$ with $w=u+iv$, such that $V(\mathbb{D}^*)=1$, then we obtain the Bergman kernel function on $\mathbb{D}^*\times\mathbb{D}^*$ given by 
$$K_{\mathbb{D}^*}(w,\eta)=\frac{1}{(1-w\overline{\eta})^2}.$$

For the unit ball $\mathbb{B}^k$ in $\mathbb{C}^k$, again we normalize the volume measure, such that $V(\mathbb{B}^k)=1$, then we know that the Bergman kernel function on $\mathbb{B}^k\times\mathbb{B}^k$ will be
$$K_{\mathbb{B}^k}(w,\eta)=\frac{1}{(1-\langle w,\eta\rangle)^{k+1}},$$
where $\langle w,\eta\rangle=\sum_{j=1}^{k}w_j\overline{\eta_j}$, for $w,\eta\in\mathbb{B}^k$.

By above consideration, for the product model $\mathbb{B}^k\times\mathbb{D}^*\times\cdots\times\mathbb{D}^*$, if we use the notation $\tilde{w}=(w_1,\dots,w_k)$ (when $w\in\mathbb{C}^n$), then by the formula for the Bergman kernel function on product domain, we obtain
$$K_{ \mathbb{B}^k\times\mathbb{D}^*\times\cdots\times\mathbb{D}^*}(w,\eta)=\frac{1}{(1-\langle \tilde{w},\tilde{\eta}\rangle)^{k+1}}\cdot\frac{1}{(1-w_{k+1}\overline{\eta_{k+1}})^2}\cdot\cdots\cdot\frac{1}{(1-w_{n}\overline{\eta_{n}})^2}.$$
Since we have the biholomorphism $G: \mathbb{B}^k\times\mathbb{D}^*\times\cdots\times\mathbb{D}^*\to\mathbb{H}^n_{k}$, by the transformation formula for the Bergman kernel functions, we obtain
\begin{equation}
\begin{split}
K_{\mathbb{H}_k^n}(z,\zeta)
&=K_{\mathbb{H}_k^n}(G(w),G(\eta))\\
&=\frac{1}{\det\mathcal{J}_G^{\mathbb{C}}(w)\overline{\det\mathcal{J}_G^{\mathbb{C}}(\eta)}(1-\langle \tilde{w},\tilde{\eta}\rangle)^{k+1}\prod_{j=k+1}^{n}(1-w_j\overline{\eta_j})^2},
\end{split}
\end{equation}
for $z,\zeta\in\mathbb{H}_k^n$ and $w,\eta\in\mathbb{B}^k\times\mathbb{D}^*\times\cdots\times\mathbb{D}^*$.\\

\section{estimates for the kernels}

\subsection{The kernel on $\mathbb{B}^k$}
For $-1<\alpha<0$, we follow the idea in \cite{R} to obtain an estimate for the integral
\[
J_{\alpha}(w)=\int_{\mathbb{B}^k}\frac{(1-\abs{\eta}^2)^{\alpha}dV(\eta)}{\abs{1-\langle w,\eta\rangle}^{k+1}},
\]
where $w\in\mathbb{B}^k$ (the restriction $\alpha>-1$ makes the integral convergent). First of all, we need the following.

\begin{lemma}
Let $\nu=(\nu_1,\dots,\nu_k)$ be a multi-index, and let $\sigma$ be the normalized area measure on $\mathbb{S}^k$ (the boundary $\partial\mathbb{B}^k$), i.e. $\sigma(\mathbb{S}^k)=1$. Then 
\[
\int_{\mathbb{S}^k}\abs{\xi^{\nu}}^2d\sigma(\xi)=\frac{(k-1)!(\nu)!}{(\abs{\nu}+k-1)!}.
\]
\end{lemma}

\begin{proof}
Let $dm$ denote the standard volume form on the Euclidean space, then we have the integral
\begin{equation}
\begin{split}
\int_{\mathbb{C}^k}\abs{z^{\nu}}^2\exp(-\abs{z}^2)dm(z)
&=\prod_{j=1}^k\int_{\mathbb{C}}\abs{z_j}^{2\nu_j}\exp(-\abs{z_j}^2)dm(z_j)\\
&=\prod_{j=1}^k(2\pi)\int_{0}^{\infty}r^{2\nu_j}e^{-r^2}rdr\\
&=\prod_{j=1}^k\pi(\nu_j)!\\
&=\pi^k(\nu)!.
\end{split}
\end{equation}
On the other hand, identifying $\mathbb{C}^k=\mathbb{R}^{2k}$, and using the polar coordinates $z=r\xi$ and $dm(z)=2kr^{2k-1}\frac{\pi^k}{k!}drd\sigma(\xi)$, we have
\begin{equation}
\begin{split}
\int_{\mathbb{C}^k}\abs{z^{\nu}}^2\exp(-\abs{z}^2)dm(z)
&=2k\cdot\frac{\pi^k}{k!}\int_{0}^{\infty}r^{2\abs{\nu}+2k-1}e^{-r^2}dr\int_{\mathbb{S}^k}\abs{\xi^{\nu}}^2d\sigma(\xi)\\
&=\frac{\pi^k}{(k-1)!}\cdot(\abs{\nu}+k-1)!\int_{\mathbb{S}^k}\abs{\xi^{\nu}}^2d\sigma(\xi).
\end{split}
\end{equation}
Hence, comparing $(3.1)$ and $(3.2)$, we have $\int_{\mathbb{S}^k}\abs{\xi^{\nu}}^2d\sigma(\xi)=\frac{(k-1)!(\nu)!}{(\abs{\nu}+k-1)!}$.
\end{proof}

With the lemma above, we can show the estimate for the integral $J_{\alpha}(w)$.

\begin{lemma}
For $-1<\alpha<0$, we have $J_{\alpha}(w)\sim(1-\abs{w}^2)^{\alpha}$, for any $w\in\mathbb{B}^k$.
\end{lemma}

\begin{proof}
For $w,\eta\in\mathbb{B}^k$, since $\abs{\langle w,\eta\rangle}<1$, we have the expansion
\begin{equation}
\begin{split}
\frac{1}{\abs{1-\langle w,\eta\rangle}^{k+1}}
&=\frac{1}{(1-\langle w,\eta\rangle)^{\frac{k+1}{2}}}\cdot\frac{1}{(1-\overline{\langle w,\eta\rangle})^{\frac{k+1}{2}}}\\
&=\sum_{n=0}^{\infty}\frac{\Gamma(n+\frac{k+1}{2})}{\Gamma(n+1)\Gamma(\frac{k+1}{2})}\langle w,\eta\rangle^n\sum_{m=0}^{\infty}\frac{\Gamma(m+\frac{k+1}{2})}{\Gamma(m+1)\Gamma(\frac{k+1}{2})}\overline{\langle w,\eta\rangle}^m.
\end{split}
\end{equation}
Substitute the expansion into $J_{\alpha}(w)$, and integrate term by term. By the rotational symmetries on $\mathbb{B}^k$, we obtain
\[
J_{\alpha}(w)=\sum_{n=0}^{\infty}\frac{\Gamma(n+\frac{k+1}{2})^2}{\Gamma(n+1)^2\Gamma(\frac{k+1}{2})^2}\int_{\mathbb{B}^k}\abs{\langle w,\eta\rangle}^{2n}(1-\abs{\eta}^2)^{\alpha}dV(\eta).
\]
Since $dV(\eta)$ is unitary invariant, if we apply a unitary transformation $U$ to the integral above, such that $U(w)=(\abs{w},0,\dots,0)$, then we have
\begin{equation}
\begin{split}
J_{\alpha}(w)
&=\sum_{n=0}^{\infty}\frac{\Gamma(n+\frac{k+1}{2})^2}{\Gamma(n+1)^2\Gamma(\frac{k+1}{2})^2}\int_{\mathbb{B}^k}\abs{w}^{2n}\abs{\eta_1}^{2n}(1-\abs{\eta}^2)^{\alpha}dV(\eta)\\
&=\sum_{n=0}^{\infty}\frac{\Gamma(n+\frac{k+1}{2})^2}{\Gamma(n+1)^2\Gamma(\frac{k+1}{2})^2}\abs{w}^{2n}\int_{0}^{1}(2k)r^{2n+2k-1}(1-r^2)^{\alpha}dr\int_{\mathbb{S}^k}\abs{\xi_1}^{2n}d\sigma(\xi)\\
&=\sum_{n=0}^{\infty}\frac{\Gamma(n+\frac{k+1}{2})^2}{\Gamma(n+1)^2\Gamma(\frac{k+1}{2})^2}\frac{k!n!}{(n+k-1)!}\abs{w}^{2n}\int_{0}^{1}\rho^{n+k-1}(1-\rho)^{\alpha}d\rho\\
&\sim\sum_{n=0}^{\infty}\abs{w}^{2n}B(\alpha+1,n+k)\\
&=\sum_{n=0}^{\infty}\frac{\Gamma(\alpha+1)\Gamma(n+k)}{\Gamma(n+k+\alpha+1)}\abs{w}^{2n}\\
&\sim\sum_{n=0}^{\infty}\frac{\Gamma(n-\alpha)}{\Gamma(n+1)\Gamma(-\alpha)}\abs{w}^{2n}=(1-\abs{w}^2)^{\alpha}.
\end{split}
\end{equation}
Here, for the second line, we use the polar coordinate $\eta_1=r\xi_1$. For the third line, we apply the previous lemma and the substutition $\rho=r^2$. From the fourth line through the last line, we apply the basic properties of the Beta function and the stirling's formula to estimate the Gamma functions (as $n\to\infty$). The last equality holds, since $\alpha<0$.
\end{proof}

\subsection{The kernel on $\mathbb{D}^*$}
Similarly, for $-1<\alpha<0$ and $\beta>-2$, we modify the idea used above to obtain an estimate for the integral
\[
I_{\alpha,\beta}(w)=\int_{\mathbb{D}^*}\frac{(1-\abs{\eta}^2)^{\alpha}\abs{\eta}^{\beta}dV(\eta)}{\abs{1-w\overline{\eta}}^2},
\]
where $w\in\mathbb{D}^*$ (again, the restrictions $\alpha>-1$ and $\beta>-2$ make the integral convergent).

\begin{lemma}
For $-1<\alpha<0$ and $\beta>-2$, we have $I_{\alpha,\beta}(w)\sim(1-\abs{w}^2)^{\alpha}$, for any $w\in\mathbb{D}^*$. In addition, when $\beta\le 0$, we have $I_{\alpha,\beta}(w) \lesssim (1-\abs{w}^2)^{\alpha}\abs{w}^{\beta}$, for $w\in\mathbb{D}^*$.
\end{lemma}

\begin{proof}
As before, for $w,\eta\in\mathbb{D}^*$, we expand the kernel function
\[
\frac{1}{\abs{1-w\overline{\eta}}^2}=\sum_{n=0}^{\infty}(w\overline{\eta})^n\sum_{m=0}^{\infty}(\overline{w}\eta)^m.
\]
Substitute the expansion back to the integral, and integrate term by term. By the rotational symmetry on $\mathbb{D}^*$, we obtain
\begin{equation}
\begin{split}
I_{\alpha,\beta}(w)
&=\sum_{n=0}^{\infty}\int_{\mathbb{D}^*}\abs{w\eta}^{2n}(1-\abs{\eta}^2)^{\alpha}\abs{\eta}^{\beta}dV(\eta)\\
&=\sum_{n=0}^{\infty}\abs{w}^{2n}\int_{0}^{1}2r^{2n+\beta+1}(1-r^2)^{\alpha}dr\\
&=\sum_{n=0}^{\infty}\abs{w}^{2n}\int_{0}^{1}\rho^{n+\frac{\beta}{2}}(1-\rho)^{\alpha}d\rho\\
&=\sum_{n=0}^{\infty}\abs{w}^{2n}B(\alpha+1, n+\frac{\beta}{2}+1)\\
&=\sum_{n=0}^{\infty}\frac{\Gamma(\alpha+1)\Gamma(n+\frac{\beta}{2}+1)}{\Gamma(n+\frac{\beta}{2}+\alpha+2)}\abs{w}^{2n}\\
&\sim\sum_{n=0}^{\infty}\frac{\Gamma(n-\alpha)}{\Gamma(n+1)\Gamma(-\alpha)}\abs{w}^{2n}=(1-\abs{w}^2)^{\alpha}.
\end{split}
\end{equation}
Again, similar to the previous lemma, we use the polar coordinate $\eta=re^{i\theta}$, then apply the substutition $\rho=r^2$. By the basic properties of the Beta function and the stirling's formula, we obtain the asymtotic behavior for the Gamma functions (as $n\to\infty$). The last equality holds, since $\alpha<0$.

When $\beta\le 0$, since $w\in\mathbb{D}^*$, we have $\abs{w}^{\beta}\ge 1$. Hence, for any $w\in\mathbb{D}^*$, we have $I_{\alpha,\beta}(w) \lesssim (1-\abs{w}^2)^{\alpha}\abs{w}^{\beta}$ by above argument.
\end{proof}
\

\section{the schur's test}

For any domain $D$, to show the $L^p$ boundedness of the Bergman projection $P_D$, we can show a stronger statement, namely the $L^p$ boundedness of the integral operator $\abs{P_D}$ associated to the kernel $\abs{K_D}$. Here $\abs{K_D}$ is the absolute value of the Bergman kernel $K_D$ for $D$. By this consideration, we need the following version of Schur's test, which can be found in \cite{HKZ}.

\begin{theorem}[Schur's test]
Suppose $X$ is measure space with a positive measure $\mu$. Let $T(x,y)$ be a positive measurable function on $X\times X$, and let $T$ be the integral operator associated to the kernel function $T(x,y)$.

Given $p,q\in(1,\infty)$ with $\frac{1}{p}+\frac{1}{q}=1$, if there exists a strictly positive function $h$ a.e. on $X$ and a $M>0$, such that
\begin{enumerate}
\item
$\int_{X}T(x,y)h(y)^qd\mu(y)\le Mh(x)^q$, for a.e. $x\in X$, and
\item
$\int_{X}T(x,y)h(x)^pd\mu(x)\le Mh(y)^p$, for a.e. $y\in X$. 
\end{enumerate}
Then $T$ is bounded on $L^p(X,d\mu)$ with $\|T\|\le M$.
\end{theorem}

\begin{proof}
Let $f\in L^p(X,d\mu)$, by H\"{o}lder's inequality and (1), we have
\begin{equation}
\begin{split}
\abs{Tf(x)}
&\le \int_X T(x,y)\abs{f(y)}d\mu(y)\\
&\le \Big(\int_X T(x,y)h(y)^qd\mu(y)\Big)^{\frac{1}{q}}\Big(\int_X T(x,y)h(y)^{-p}\abs{f(y)}^pd\mu(y)\Big)^{\frac{1}{p}}\\
&\le M^{\frac{1}{q}}h(x)\Big(\int_X T(x,y)h(y)^{-p}\abs{f(y)}^pd\mu(y)\Big)^{\frac{1}{p}},
\end{split}
\end{equation}
for a.e. $x\in X$. So, by Fubini's theorem and (2), we have
\begin{equation}
\begin{split}
\int_X\abs{Tf(x)}^pd\mu(x)
&\le M^{\frac{p}{q}}\int_X h(x)^pd\mu(x)\int_X T(x,y)h(y)^{-p}\abs{f(y)}^pd\mu(y)\\
&=M^{\frac{p}{q}}\int_X T(x,y)h(x)^pd\mu(x)\int_X h(y)^{-p}\abs{f(y)}^pd\mu(y)\\
&\le M^{\frac{p}{q}+1}\int_X \abs{f(y)}^pd\mu(y)\\
&=M^p\|f\|^p.
\end{split}
\end{equation}
This completes the proof.
\end{proof}
\

\section{proof of theorem 1.2}
With the above preliminaries, we now can present the proof of theorem 1.2.
\begin{proof}
According to theorem 4.1 (Schur's test), if we take $X=\mathbb{B}^k\times\mathbb{D}^*\times\cdots\times\mathbb{D}^*$, $d\mu(\eta)=\abs{\det J_{G}^{\mathbb{C}}(\eta)}^2dV(\eta)$, and $T(w,\eta)=\abs{K_{\mathbb{H}_k^n}(G(w),G(\eta))}$ from (2.1), then for any measurable function $f$ defined on $\mathbb{H}_k^n$, we have
\[
T(f\circ G)=\abs{P_{\mathbb{H}_k^n}}(f)\circ G.
\]
So, for a given $p\in[1,\infty)$, $T$ is bounded on $L^p(X,d\mu)$ if and only if $\abs{P_{\mathbb{H}_k^n}}$ is bounded on $L^p(\mathbb{H}^n_{k})$. Now let
\[
h(\eta)=(1-\abs{\tilde{\eta}}^2)^s[(1-\abs{\eta_{k+1}}^2)\cdots(1-\abs{\eta_n}^2)]^s\abs{\eta_{k+1}}^{t_{k+1}}\cdots\abs{\eta_n}^{t_n},
\]
where $s,t_{k+1},\dots,t_n$ are real numbers to be determined later and $\tilde{\eta}=(\eta_1,\dots,\eta_k)$ as before. We need to verify the conditions (1) and (2) in theorem 4.1 to conclude the $L^p$ boundedness of $T$, for the given $p,q\in(1,\infty)$ with $\frac{1}{p}+\frac{1}{q}=1$. 

For condition (1), by lemma 3.2 and lemma 3.3, we have
\begin{equation}
\begin{split}
T(h^q)(w)
&=\int_X T(w,\eta)h(\eta)^qd\mu(\eta)\\
&=\int_{\mathbb{B}^k}\frac{(1-\abs{\tilde{\eta}}^2)^{sq}dV(\tilde{\eta})}{\abs{1-\langle\tilde{w},\tilde{\eta}\rangle}^{k+1}}\prod_{j=k+1}^n\int_{\mathbb{D}^*}\frac{(1-\abs{\eta_{j}}^2)^{sq}\abs{\eta_{j}}^{t_{j}q+j-1}dV(\eta_{j})}{\abs{w_{j}}^{j-1}\abs{1-w_{j}\overline{\eta_{j}}}^2}\\
&\le M(1-\abs{\tilde{w}}^2)^{sq}\prod_{j=k+1}^n(1-\abs{w_{j}}^2)^{sq}\abs{w_{j}}^{t_{j}q}\\
&=Mh(w)^q,
\end{split}
\end{equation}
for some $M>0$, provided
\begin{equation}
\begin{split}
&-1<sq<0,\\
&-2<t_{k+1}q+k\le 0,\\
&\vdots\\
&-2<t_nq+n-1\le 0.
\end{split}
\end{equation}
For condition (2), similar argument shows,
\begin{equation}
\begin{split}
T^*(h^p)(\eta)
&:=\int_X T(w,\eta)h(w)^pd\mu(w)\\
&=\int_{\mathbb{B}^k}\frac{(1-\abs{\tilde{w}}^2)^{sp}dV(\tilde{w})}{\abs{1-\langle\tilde{w},\tilde{\eta}\rangle}^{k+1}}\prod_{j=k+1}^n\int_{\mathbb{D}^*}\frac{(1-\abs{w_{j}}^2)^{sp}\abs{w_{j}}^{t_{j}p+j-1}dV(w_{j})}{\abs{\eta_{j}}^{j-1}\abs{1-w_{j}\overline{\eta_{j}}}^2}\\
&\le M(1-\abs{\tilde{\eta}}^2)^{sp}\prod_{j=k+1}^n(1-\abs{\eta_{j}}^2)^{sp}\abs{\eta_{j}}^{t_{j}p}\\
&=Mh(\eta)^p,
\end{split}
\end{equation}
for some $M>0$, provided
\begin{equation}
\begin{split}
&-1<sp<0,\\
&-2<t_{k+1}p+k\le 0,\\
&\vdots\\
&-2<t_np+n-1\le 0.
\end{split}
\end{equation}
From (5.2) and (5.4), we have
\[
s\in\Big(-\frac{1}{q},0\Big)\bigcap\Big(-\frac{1}{p},0\Big),
\]
\[
t_{k+1}\in\Big(-\frac{k+2}{q},-\frac{k}{q}\Big]\bigcap\Big(-\frac{k+2}{p},-\frac{k}{p}\Big],
\]
\[
\vdots
\]
\[t_n\in\Big(-\frac{n+1}{q},-\frac{n-1}{q}\Big]\bigcap\Big(-\frac{n+1}{p},-\frac{n-1}{p}\Big].
\]
So the real numbers $s,t_{k+1},\dots,t_{n}$ exist when
\[
-\frac{k}{q}>-\frac{k+2}{p},\ -\frac{k}{p}>-\frac{k+2}{q},
\]
\[
\vdots
\]
\[
-\frac{n-1}{q}>-\frac{n+1}{p},\ -\frac{n-1}{p}>-\frac{n+1}{q},
\]
or equivalently, when $\frac{2n}{n+1}<p<\frac{2n}{n-1}$. This proves the first part of the theorem.\\

To show the unboundedness of $P_{\mathbb{H}_k^n}$ for $1\le p\le \frac{2n}{n+1}$, we first look at the case $p=\frac{2n}{n+1}$.

For $j=1,2,\dots$, let $a_j=\frac{1}{j^j}$. Define $g(r)=r^{\frac{1}{j}-(n+1)}$, when $r\in(a_{j+1},a_j]$. Then we have a function $g$ defined on $(0,1]$. Now, if we look at the sequence $\{f_m\}_{m=1}^{\infty}$ given by
\[
f_m(z)=\left\{ \begin{array}{rcl}
g(\abs{z_n})\big(\frac{\overline{z_n}}{\abs{z_n}}\big)^{n-1}, & \abs{z_n}\in(a_{m+1},1]\\
0, & \abs{z_n}\in[0,a_{m+1}]
\end{array}\right.
\]
for $z\in\mathbb{H}_k^n$, then we have
\begin{equation}
\begin{split}
\|f_m\|_{L^{\frac{2n}{n+1}}(\mathbb{H}_k^n)}^{\frac{2n}{n+1}}
&=\int_{\mathbb{B}^k\times\mathbb{D}^*\times\cdots\times\mathbb{D}^*}\abs{f_m(G(w))}^{\frac{2n}{n+1}}\abs{w_{k+1}^k\cdots w_n^{n-1}}^2dV(w)\\
&\le \int_{a_{m+1}}^1g(r)^{\frac{2n}{n+1}}r^{2n-1}dr\\
&\le \sum_{j=1}^mj(a_j^{\frac{2n}{j(n+1)}}-a_{j+1}^{\frac{2n}{j(n+1)}})\\
&\le \sum_{j=1}^{\infty}j(a_j^{\frac{2n}{j(n+1)}}-a_{j+1}^{\frac{2n}{j(n+1)}})\\
&\lesssim \sum_{j=1}^{\infty}j^{-1-\frac{1}{2(n+1)}}<\infty,
\end{split}
\end{equation}
by the limit comparison test. Since each $f_m$ is bounded on $\mathbb{H}_k^n$, we see that $f_m\in L^2(\mathbb{H}_k^n)$. So we have constructed a sequence $\{f_m\}_{m=1}^{\infty}\subset L^2(\mathbb{H}_k^n)\cap L^{\frac{2n}{n+1}}(\mathbb{H}_k^n)$, with $\|f_m\|_{L^{\frac{2n}{n+1}}(\mathbb{H}_k^n)}$ bounded by a constant for all $m$.

On the other hand, for $w\in\mathbb{B}^k\times\mathbb{D}^*\times\cdots\times\mathbb{D}^*$, by the reproducing property and the rotational symmetry, we have
\begin{equation}
\begin{split}
\abs{P_{\mathbb{H}_k^n}(f_m)(G(w))}
&=\abs{\int_{\mathbb{B}^k\times\mathbb{D}^*\times\cdots\times\mathbb{D}^*}\frac{f_m(G(\eta))\eta_{k+1}^k\cdots \eta_n^{n-1}dV(\eta)}{w_{k+1}^k\cdots w_n^{n-1}(1-\langle \tilde{w},\tilde{\eta}\rangle)^{k+1}\prod_{j=k+1}^{n}(1-w_j\overline{\eta_j})^2}}\\
&=\frac{1}{\abs{w_n}^{n-1}}\abs{\int_{\mathbb{D}^*}\frac{f_m(G(\eta))\eta_n^{n-1}dV(\eta_n)}{(1-w_n\overline{\eta_n})^2}}\\
&=\frac{1}{\abs{w_n}^{n-1}}\int_{a_{m+1}<\abs{\eta_n}\le 1}g(\abs{\eta_n})\abs{\eta_n}^{n-1}dV(\eta_n)\\
&\ge \int_{a_{m+1}}^{1}g(r)r^ndr\\
&=\sum_{j=1}^mj(a_j^{\frac{1}{j}}-a_{j+1}^{\frac{1}{j}}).
\end{split}
\end{equation}
Therefore, since $V(\mathbb{H}_k^n)$ is finite,
\[
\lim_{m\to\infty}\|P_{\mathbb{H}_k^n}(f_m)\|_{L^{\frac{2n}{n+1}}(\mathbb{H}_k^n)} \ge V(\mathbb{H}_k^n)^{\frac{n+1}{2n}}\sum_{j=1}^{\infty}j(a_j^{\frac{1}{j}}-a_{j+1}^{\frac{1}{j}})\gtrsim \sum_{j=1}^{\infty}j^{-1}=\infty,
\]
again by the limit comparison test. So $P_{\mathbb{H}_k^n}$ is not bounded on $L^{\frac{2n}{n+1}}(\mathbb{H}_k^n)$.

For $1\le p< \frac{2n}{n+1}$, we take the same sequence $\{f_m\}_{m=1}^{\infty}$ given above. Notice that, by H\"{o}lder's inequality and (5,5), we have
\[
\|f_m\|_{L^p(\mathbb{H}_k^n)}\le C\|f_m\|_{L^{\frac{2n}{n+1}}(\mathbb{H}_k^n)}<C'<\infty,
\]
for some constants $C,\ C'>0$. So we agian have a sequence $\{f_m\}_{m=1}^{\infty}\subset L^2(\mathbb{H}_k^n)\cap L^{p}(\mathbb{H}_k^n)$, with $\|f_m\|_{L^{p}(\mathbb{H}_k^n)}$ bounded by a constant for all $m$. However, from (5.6), we see that
\[
\lim_{m\to\infty}\|P_{\mathbb{H}_k^n}(f_m)\|_{L^{p}(\mathbb{H}_k^n)} \ge V(\mathbb{H}_k^n)^{\frac{1}{p}}\sum_{j=1}^{\infty}j(a_j^{\frac{1}{j}}-a_{j+1}^{\frac{1}{j}})\gtrsim \sum_{j=1}^{\infty}j^{-1}=\infty.
\]
Hence, $P_{\mathbb{H}_k^n}$ is not bounded on $L^{p}(\mathbb{H}_k^n)$ for $1\le p\le \frac{2n}{n+1}$. This proves the second part of the theorem.\\

For our standard model
$$\mathbb{H}^n_{\{k_j\}}=\{z\in\mathbb{C}^n|\max_{1\le j \le l}\abs{\tilde{z_j}}<\abs{z_{k+1}}<\cdots<\abs{z_n}<1\},$$ 
we again transfer it to a product domian via the biholomorphism $\tilde{F}:\mathbb{H}^n_{\{k_j\}}\to\mathbb{B}^{k_1}\times\cdots\times\mathbb{B}^{k_l}\times\mathbb{D}^*\times\cdots\times\mathbb{D}^*$ ($n-k$ copies of $\mathbb{D}^*$) given by
\[
\tilde{F}(z_1,\dots,z_n)=(\frac{z_1}{z_{k+1}},\dots,\frac{z_k}{z_{k+1}},\frac{z_{k+1}}{z_{k+2}},\dots,\frac{z_{n-1}}{z_n},z_n),
\]
with its inverse
\[
\tilde{G}(w_1,\dots,w_n)=(w_1(w_{k+1}\cdots w_n),\dots,w_k(w_{k+1}\cdots w_n),(w_{k+1}\cdots w_n),\dots,(w_{n-1}w_n),w_n)
\]
and $\det\mathcal{J}_{\tilde{G}}^{\mathbb{C}}(w)=w_{k+1}^k\cdots w_n^{n-1}$. Similarly, we have
\[
K_{\mathbb{H}_{\{k_j\}}^n}(z,\zeta)=\frac{1}{\det\mathcal{J}_{\tilde{G}}^{\mathbb{C}}(w)\overline{\det\mathcal{J}_{\tilde{G}}^{\mathbb{C}}(\eta)}\prod_{j=1}^{l}(1-\langle \tilde{w_j},\tilde{\eta_j}\rangle)^{k_j+1}\prod_{j=k+1}^{n}(1-w_j\overline{\eta_j})^2},
\]
for $w,\eta\in\mathbb{B}^{k_1}\times\cdots\times\mathbb{B}^{k_l}\times\mathbb{D}^*\times\cdots\times\mathbb{D}^*$ and $(z,\zeta)=(\tilde{G}(w),\tilde{G}(\eta))\in\mathbb{H}_{\{k_j\}}^n\times\mathbb{H}_{\{k_j\}}^n$. To apply theorem 4.1 (Schur's test) to conclude the $L^p$ boundedness of the operator $\abs{P_{\mathbb{H}_{\{k_j\}}^n}}$, we only need to modify the positive function $h$ by replacing the factor $(1-\abs{\tilde{\eta}}^2)^s$ by $\prod_{j=1}^l(1-\abs{\tilde{\eta_j}}^2)^s$. It is easy to see the range $(\frac{2n}{n+1},\frac{2n}{n-1})$ for $p$ will not change. Also, the same sequence $\{f_m\}_{m=1}^{\infty}$ works well in our standard model $\mathbb{H}^n_{\{k_j\}}$. Therefore, $P_{\mathbb{H}_{\{k_j\}}^n}$ is not bounded on $L^{p}(\mathbb{H}_{\{k_j\}}^n)$ for $1\le p\le \frac{2n}{n+1}$. This completes the proof. 
\end{proof}

\begin{remark}
As we seen in the proof, the unboundedness of $P_{\mathbb{H}_{\{k_j\}}^n}$ is still valid when $0<p<1$.
\end{remark}

\

\section{proof of theorem 1.1}
\begin{proof}
For the bounded Hartogs domain
\[
\mathbb{H}^n_{\{k_j,\phi_j\}}=\{z\in\mathbb{C}^n|\max_{1\le j \le l}\abs{\phi_j(\tilde{z_j})}<\abs{z_{k+1}}<\cdots<\abs{z_{n}}<1\},
\]
we define a biholomorphism $\Phi:\mathbb{H}^n_{\{k_j,\phi_j\}}\to\mathbb{H}^n_{\{k_j\}}$ by
\[
\Phi(z)=(\phi_1(\tilde{z_1}),\dots,\phi_l(\tilde{z_l}),z_{k+1},\dots,z_n),
\]
with inverse
\[
\Phi^{-1}(w)=(\phi^{-1}_1(\tilde{w_1}),\dots,\phi^{-1}_l(\tilde{w_l}),w_{k+1},\dots,w_n).
\]
By a direct computation, it is not difficult to see, for $z\in\mathbb{H}^n_{\{k_j,\phi_j\}}$,
\[
\det\mathcal{J}_{\Phi}^{\mathbb{C}}(z)=\prod_{j=1}^l\det\mathcal{J}_{\phi_j}^{\mathbb{C}}(\tilde{z_j}).
\]
For each $j$, we have the biholomorphism $\phi_j:\Omega_j\to\mathbb{B}^{k_j}$, with both $\Omega_j$ and $\mathbb{B}^{k_j}$ being smooth and bounded. Since $\mathbb{B}^{k_j}$ is strongly pseudoconvex, it satisfies condition R. By Bell's extension theorem (see \cite{B}), $\phi_j$ and $\phi_j^{-1}$ extend smoothly to the boundaries. So, we can find two positive real numbers $c_j$ and $d_j$, such that for any $\tilde{z_j}\in\Omega_j$,
\[
0<c_j\le \abs{\det\mathcal{J}_{\phi_j}^{\mathbb{C}}(\tilde{z_j})}\le d_j.
\]
Hence, if $c=\prod_{j=1}^lc_j$ and $d=\prod_{j=1}^ld_j$, for $z\in\mathbb{H}^n_{\{k_j,\phi_j\}}$, we have
\[
0<c\le \abs{\det\mathcal{J}_{\Phi}^{\mathbb{C}}(z)}\le d.
\]
Suppose for some $p\in[1,\infty)$, the Bergman projection $P_{\mathbb{H}^n_{\{k_j,\phi_j\}}}$ is bounded on $L^p(\mathbb{H}^n_{\{k_j,\phi_j\}})$, then for $f\in L^p(\mathbb{H}^n_{\{k_j\}})$, we have
\begin{equation}
\begin{split}
\|P_{\mathbb{H}^n_{\{k_j\}}}(f)\|^p_{L^p(\mathbb{H}^n_{\{k_j\}})}
&=\int_{\mathbb{H}^n_{\{k_j,\phi_j\}}}\abs{P_{\mathbb{H}^n_{\{k_j\}}}(f)(\Phi(z))}^p\abs{\det\mathcal{J}_{\Phi}^{\mathbb{C}}(z)}^2dV(z)\\
&=\int_{\mathbb{H}^n_{\{k_j,\phi_j\}}}\abs{P_{\mathbb{H}^n_{\{k_j,\phi_j\}}}(f\circ\Phi\cdot\det\mathcal{J}_{\Phi}^{\mathbb{C}})(z)}^p\abs{\det\mathcal{J}_{\Phi}^{\mathbb{C}}(z)}^{2-p}dV(z)\\
&\le \max\{c^{2-p},d^{2-p}\}\|P_{\mathbb{H}^n_{\{k_j,\phi_j\}}}(f\circ\Phi\cdot\det\mathcal{J}_{\Phi}^{\mathbb{C}})\|^p_{L^p(\mathbb{H}^n_{\{k_j,\phi_j\}})}\\
&\le C\max\{c^{2-p},d^{2-p}\}\|f\circ\Phi\cdot\det\mathcal{J}_{\Phi}^{\mathbb{C}}\|^p_{L^p(\mathbb{H}^n_{\{k_j,\phi_j\}})}\\
&\le Cc^{-\abs{p-2}}d^{\abs{p-2}}\|f\|^p_{L^p(\mathbb{H}^n_{\{k_j\}})},
\end{split}
\end{equation}
for some $C>0$. So the Bergman projection $P_{\mathbb{H}^n_{\{k_j\}}}$ is bounded on $L^p(\mathbb{H}^n_{\{k_j\}})$.

Conversely, if we apply the same argument to $\Phi^{-1}:\mathbb{H}^n_{\{k_j\}}\to\mathbb{H}^n_{\{k_j,\phi_j\}}$, we see that the $L^p$ boundedness of $P_{\mathbb{H}^n_{\{k_j\}}}$ will imply the $L^p$ boundedness of $P_{\mathbb{H}^n_{\{k_j,\phi_j\}}}$. So, for a given $p\in[1,\infty)$, $P_{\mathbb{H}^n_{\{k_j,\phi_j\}}}$ is bounded on $L^p(\mathbb{H}^n_{\{k_j,\phi_j\}})$ if and only if $P_{\mathbb{H}^n_{\{k_j\}}}$ is bounded on $L^p(\mathbb{H}^n_{\{k_j\}})$.

Since the Bergman projection is self-adjoint, given $p,q\in(1,\infty)$ with $\frac{1}{p}+\frac{1}{q}=1$, the $L^p$ boundedness of $P_{\mathbb{H}^n_{\{k_j\}}}$ will imply the $L^q$ boundedness of $P_{\mathbb{H}^n_{\{k_j\}}}$. Hence, by theorem 1.2, for $1\le p<\infty$, $P_{\mathbb{H}^n_{\{k_j\}}}$ is bounded on $L^p(\mathbb{H}^n_{\{k_j\}})$ if and only if $\frac{2n}{n+1}<p<\frac{2n}{n-1}$. Therefore, $P_{\mathbb{H}^n_{\{k_j,\phi_j\}}}$ is bounded on $L^p(\mathbb{H}^n_{\{k_j,\phi_j\}})$ if and only if $p$ is in the range $(\frac{2n}{n+1},\frac{2n}{n-1})$. This completes the proof.
\end{proof}

\

\section{Concluding Remarks}
We have shown that the $L^p$ boundedness of the Bergman projection for a class of bounded Hartogs domains is valid when $p$ is in some range depending only on the dimension. And the range is sharp for $p\in [1,\infty)$. However, the restriction on the form of this class of domains is quite special. It should be pointed out that there could be other generalization of the Hartogs triangle in higher dimensions. We will study this direction in the future. Also, it will be very interesting to fine more techniques to deal with other types of Hartogs domains and other integral operators for Hartogs domains.

\


\bibliographystyle{amsplain}

\end{document}